\documentclass{amsart}
\usepackage[T1]{fontenc}   

\usepackage[latin1]{inputenc}
\usepackage[english]{babel}
\usepackage[T1]{fontenc}
\usepackage{amsmath,amsthm,amsfonts,amssymb,booktabs,latexsym}
\usepackage{bbm}

\usepackage{graphicx}
\usepackage{epsf}
\usepackage{amssymb}
\usepackage{times}
\usepackage{latexsym,amscd}
\usepackage{graphicx}
\usepackage{epsf}
\usepackage{amssymb}
\usepackage{times}
\usepackage{latexsym,amscd,amsmath,amsthm,amssymb}

\usepackage{mathrsfs} 

\newtheorem{theorem}{Theorem}[section]
\newtheorem{lemma}[theorem]{Lemma}
\newtheorem{proposition}[theorem]{Proposition}
\theoremstyle{definition}

\numberwithin{equation}{section}

\theoremstyle{definition}
\newtheorem{definition}[theorem]{Definition}
\newtheorem{example}[theorem]{Example}

\newtheorem{corollary}[theorem]{Corollary}

\theoremstyle{remark}

\newcommand {\naturals} {\mathbb{N}}

\newcommand {\F} {\mathcal{F}}

\newcommand {\G} {\mathcal{G}}
\newcommand {\U} {\mathcal{U}}
\newcommand {\V} {\mathcal{V}}
\newcommand {\la} {\leftarrow}

\begin{document}

\noindent                                             
\begin{picture}(150,36)                               
\put(5,20){\tiny{Submitted to}}                       
\put(5,7){\textbf{Topology Proceedings}}              
\put(0,0){\framebox(140,34){}}                        
\put(2,2){\framebox(136,30){}}                        
\end{picture}                                        
\vspace{0.5in}

\renewcommand{\bf}{\bfseries}
\renewcommand{\sc}{\scshape}
\vspace{0.5in}

\title[H-closed Spaces and H-sets in the Convergence Setting]%
{H-closed Spaces and H-sets\\ in the Convergence Setting}

\author{John P Reynolds}
\address{Department of Mathematics University of Kansas;
Lawrence, Kansas 66045}
\email{jreynolds@ku.edu}



\thanks {}

\begin{abstract}
We use convergence theory as the framework for studying H-closed spaces and H-sets in topological spaces.  From this viewpoint, it becomes clear that the property of being H-closed and the property of being an H-set in a topological space are pretopological notions.  Additionally, we define a version of H-closedness for pretopological spaces and investigate the properties of such a space.
\end{abstract}

\maketitle

\section{\bf Introduction}
The early development of general topology was guided in part by the desire to develop a framework in which to discuss different notions of convergence found in analysis.  Starting with M. Fr\'{e}chet \cite{Frechet} and continuing with E. \v{C}ech \cite{Cech} and F. Hausdorff, different notions of convergence informed the axiomatizations of topological spaces and closure spaces. In 1948, G. Choquet \cite{Choquet} laid out the theory of {\it convergence spaces}, general enough to contain the classes of topological spaces and closure spaces while unifying the desired notions of convergence.  

Once an agreed-upon definition of topological space was achieved, the concept of compactness revealed itself to be deserving of much study and subsequently of generalization.  One of the most fruitful generalizations of compactness is that of an {\it H-closed space}, defined in \cite{Alexandroff Urysohn} by P. Alexandroff and P. Urysohn in 1928. One particular advantage of considering H-closed spaces is that, in contrast with compact spaces, every Hausdorff topological space can be densely embedded in an H-closed space. Much later, in 1968, N. V. Veli\v{c}ko \cite{Velicko} relativized H-closedness to subspaces by defining the {\it H-sets} of a space $X$.  In this same paper, Veli\v{c}ko gives us the tools needed to consider H-closedness and H-sets as purely convergence-theoretic properties.   In \cite{Dickman Porter 1975} R.F. Dickman and J. Porter use these tools to define the particular convergence we will use to discuss H-closed spaces and H-sets in the convergence setting.

Our first task here will be to place H-closed spaces and H-sets in the convergence theoretic framework. In section 2, we give preliminary definitions and results pertaining to H-closed spaces and H-sets in the usual topological setting.  This is followed in section 3 by the basic definitions and results necessary to consider the convergence theoretic point of view.  Particularly of interest will be the definition of pretopological spaces, which is the subcategory of convergence spaces in which we will mainly work.  At this point we will frame H-closed spaces and H-sets as pretopological notions.  In particular, theorem 3.10 points to the potential advantages of this point of view.

In section 4, we define a purely convergence-theoretic notion which parallels that of H-closedness for topological spaces.  The basic properties of the so-called {\it PHC spaces} (short for pretopologically H-closed spaces) are investigated.  Additionally, we develop a technique for constructing new PHC spaces using images of compact pretopological spaces.  

Lastly, we will discuss convergence-theoretic extensions of convergence spaces.  Much work has been done in this area, in particular by D.C. Kent and G.D. Richardson, who catalogued much of the early progress in the field in \cite{Kent Richardson 1979}.  We investiagate PHC extensions of a pretopological space $X$.  These extensions prove to be of interest in that for any pretopological space $X$, there is a PHC extension of $X$ which is projectively larger than any compactification of $X$.  This is not true of compactifications, as a convergence space $X$ does not in general have a largest compactification. 

\section{\bf H-closed Spaces and H-sets}

A Hausdorff topological space is {\it H-closed} if it is closed in every Hausdorff topological space in which it is embedded.  The following well-known characterizations of H-closed spaces are useful and will be used interchangably as the definition of H-closed.

\begin{theorem}
Let $X$ be a Hausdorff topological space.  The following are equivalent.
\begin{enumerate}
\item $X$ is H-closed,
\item Whenever $\mathcal{C}$ is an open cover of $X$, there exist $C_1, ..., C_n \in \mathcal{C}$ such that $X = \bigcup_{i=1}^n cl_X C_i$,
\item Every open filter on $X$ has nonempty adherence,
\item Every open ultrafilter on $X$ has a convergence point.
\end{enumerate}
\end{theorem}

Velicko \cite{Velicko} relativized the concept of H-closed to subspaces in the following way: If $X$ is a Hausdorff topological space and $A \subseteq X$, we say that $A$ is an {\it H-set} if whenever $\mathcal{C}$ is a cover of $A$ by open subsets of $X$, there exist $C_1,...,C_n \in \mathcal{C}$ such that $A \subseteq \bigcup_{i=1}^n cl_X C_i$.  We say that a filter $\mathcal{F}$ {\it meets} a set $A$ if $F \cap A \not= \varnothing$ for each $F \in \mathcal{F}$.  If $\mathcal{F}$ meets $A$ we will sometimes write $\mathcal{F} \# A$.  We note the following well-known characterizations of H-sets which mirror the above theorem.

\begin{proposition}
Let $X$ be a topological space and $A \subseteq X$.  The following are equivalent.
\begin{enumerate}
\item $A$ is an H-set in $X$,
\item If $\mathcal{F}$ is an open filter on $X$ which meets $A$, then $adh_X \mathcal{F} \cap A \not= \varnothing$,
\item If $\mathcal{U}$ is an open ultrafilter on $X$ which meets $A$, then $adh_X \U \cap A \not= \varnothing$.
\end{enumerate}
\end{proposition}

It is important to note that the property of H-closeness is not closed-hereditary.  Also, note that the definition of an H-set is heavily dependent on the ambient space being considered.  In particular, not every H-set is H-closed.  The following example, due to Urysohn, points to this distinction.  Recall that a space $X$ is semiregular if the regular-open subsets of $X$ form an open base.

\begin{example}
Let $X = \mathbb{N} \times \mathbb{Z} \cup \{\pm\infty\}$.  Define $U \subseteq X$ to be open if 
\begin{itemize} 
\item $+\infty \in U$ (resp. $-\infty \in U)$ implies that for some $k \in \naturals, \{(n, m) : n > k, m \in \naturals\} \subseteq U$ (resp. $\{(n,m) : n > k, -m \in \naturals\} \subseteq U)$

\item $(n, 0) \in U$ implies that for some $k \in \naturals$ $\{(n, \pm m) : m > k\} \subseteq U$ 
\end{itemize}

Then $X$ is H-closed and semiregular. Let $A = \{(n, 0) : n \in \naturals\} \cup \{+\infty\}$.  Notice that $A$ is a closed discrete subset of $X$ and that $A$ is an H-set in $X$.  However, with the subspace topology, $A \cong \mathbb{N}$, and thus us not H-closed.  
\end{example}

Both H-closed spaces and H-sets can be characterized in using the {\it $\theta$-closure}, which is also due to Velicko.  If $X$ is a topological space and $A \subseteq X$, then $cl_\theta A = \{x \in X : x \in U \in \tau(X) \text{ implies } cl_X U \cap A \not= \varnothing\}$, is the $\theta$-closure of $A$.  A subset is {\it $\theta$-closed} if it is equal to its $\theta$-closure.  Note that the $\theta$-closure is not a Kuratowski closure operator.  In particular it is not necessarily idempotent.  In Urysohn's example above, let $B = \{(n, m) : n \in \naturals, m > 0\}$.  Then, $cl_\theta B = B \cup \{(n, 0) : n \in \naturals\} \cup \{+\infty\}$.  However, $cl_\theta cl_\theta B = cl_\theta B \cup \{- \infty\}$.

In this respect, $(X, cl_\theta)$ is a closure space in the sese of \v{C}ech \cite{Cech}. We will see in section 3 that this characterizes the $\theta$-closure as a pretopological notion.  For a filter on $X$, define $adh_\theta \F = \bigcap_{F \in \F} cl_\theta F$.     We now give a characterization of H-closed spaces and H-sets in terms of $\theta$-closure.

\begin{theorem} Let $X$ be a Hausdorff topological space and $A \subseteq X$.  Then,
\begin{enumerate}
\item $X$ is H-closed if and only if whenever $\F$ is a filter on $X$, $adh_\theta \F \not=\varnothing$.
\item $A$ is an H-set in $X$ if and only if whenever $\F$ is a filter on $X$ which meets $A$, $adh_\theta \F \cap A \not= \varnothing$.
\end{enumerate}
\end{theorem}

Note that in the above theorem, $\F$ is a filter consisting of any subsets of $X$, not an open filter as in theorem 2.1. 

A function between topological spaces, $f:X \to Y$, is called {\it $\theta$-continuous} if whenever $V$ is an open neighborhood of $f(x)$, there exists and open neighborhood $U$ of $x$ such that $f[cl_X U] \subseteq cl_Y V$.  Often the notion of $\theta$-continuity is more useful than that of continuity for Hausdorff, non-regular topological spaces.  For example, for every Hausdorff space $X$, there exists an extremally disconnected, Tychonoff space $EX$, called the absolute of $X$, and a perfect, irreducible, $\theta$-continuous map $k_X : EX \to X$.  More, the absolute of $X$ is unique in a sense.  For a full treatment of absolutes, see \cite{Porter Woods}. 

\section{\bf Convergence Spaces}

For a basic reference on convergence theory, see \cite{Dolecki 2008}.  Given a relation $\xi$ between filters on $X$ and elements of $X$, we write either $\F \to_\xi x$ or $x \in \lim_\xi \F$ whenever $(\F, x) \in \xi$. If $A \subseteq X$, let $\langle A \rangle$ be the principal filter generated by $A$.  We abbreviate $\langle \{x\} \rangle$ by $\langle x \rangle$. A {\it convergence space} is a set $X$ paired with a relation $\xi$ between filters on $X$ and points of $X$ satisfying
\begin{enumerate}
\item $\langle x \rangle \to_\xi x$, and
\item if $\F \subseteq \G$ and $\F \to_\xi x$, then $\G \to_\xi x$.
\end{enumerate}

Clearly, a topological space paired with the usual topological notion of convergence in which $\F \to x$ if and only if $\mathcal{N}(x) \subseteq \F$ is an example of a convergence space.  The class of convergence structures on a set $X$ can be given a lattice structure.  We say that $\tau$ is {\it coarser than} $\xi$, written $\tau \leq \xi$ if $lim_\tau \F \supseteq \lim_\xi \F$ for each filter $\F$ on $X$.  In this case we also say that $\xi$ is finer than $\tau$.

\begin{example}
Throughout this paper, if $X$ is a topological space, let $\theta_X$  be the convergence on $X$ defined by $\F \to_{\theta_X} x$ if and only if $cl_X U \in \F$ for each open neighborhood $U$ of $x$.  If only one topological space $X$ is being considered, we will drop the subscript on $\theta$.  This type of convergence was studied extensively under the name ``almost convergence'' in \cite{Dickman Porter 1975}.  We will frequently come back to this example of a convergence space.  
\end{example}

Two filters $\F$ and $\G$ meet if $F \cap G \not= \varnothing$ for each $F \in \F$ and $G \in \G$, in which case we write $\F \# \G$.  Given a filter $\F$ on a convergence space $(X, \xi)$, the {\it adherence of $\F$} is defined to be \[adh_\xi \F = \bigcup \{\lim_\xi \G : \G \# \F\}.\]  For $A \subseteq X$, we write $adh_\xi A$ to abbreviate $adh_\xi \langle A \rangle$.  We will also define the {\it inherence} of a set $A$ by \[inh_\xi A = X \setminus adh_\xi (X \setminus A).\]  These two concepts will function as generalized versions of topological closure and interior for convergence spaces.

A convergence $\xi$ is {\it Hausdorff} if every filter has at most one limit point.

Topological spaces are now seen is a particular instance of a convergence space.  In fact, if $(X, \tau)$ is a topological space, then $adh_\tau A = cl_X A$ for any $A \subseteq X$ and $adh_\tau \F = \bigcap_{F \in \F} cl_X F$.  Two other important classes of convergence spaces are {\it pseudotopologies} and {\it pretopologies}.  If $\F$ is a filter on $X$, let $\beta \F$ denote the set of all ultrafilters on $X$ containing $\F$.  A convergence $\xi$ is a pseudotopology if $\lim_\xi \F \supseteq \bigcap \{\lim_\xi \mathcal{U} : \U \in \beta \F\}$. In \cite{Herrlich}, Herrlich, Lowen-Colebunders and Schwatz discuss the categorical advantages of working in the category of pseudotopological spaces.  We will discuss the usefulness of working with pretopological spaces to characterize H-closed space and H-sets in the next subsection.

A convergence space $(X, \xi)$ is {\it compact} if every filter on $X$ has nonempty adherence.  The following notions of compactness for filters will allow us to get at compactness of subspaces.

\begin{definition} Let $(X, \xi)$ be a convergence space, $\F$ a filter on $X$ and $A \subseteq X$.  We say that $\F$ is {\it compact at $A$} if whenever $\G$ is a filter on $X$ and $\G \# \F$, $adh_\xi \G \cap A \not= \varnothing$.  In particular, a filter $\F$ is {\it relatively compact} if $\F$ is compact at $X$.  

If $\mathcal{B}$ is a family of subsets of $X$, then $\F$ is {\it compact at $\mathcal{B}$} if whenever $\G \# \F$, $adh_\xi \G \# \mathcal{B}$.  A filter is {\it compact} if $\F$ is compact at itself.
\end{definition}

Using this definition, $A \subseteq X$ is {\it compact} if whenever $\G$ is a filter on $X$ which meets $A$, we have that $adh_\xi \G \cap A \not= \varnothing$.  Notice that for topological spaces this also characterizes the compact subspaces. 

Let $(X, \xi)$ and $(Y, \tau)$ be convergence spaces.  A function $f:(X, \xi) \to (Y, \tau)$ is {\it continuous} if $f[\lim_\xi \F] \subseteq \lim_\tau f(\F)$ for each filter $\F$ on $X$, where $f(\F)$ is the filter generated by $\{f[F] : F \in \F\}$.  Notice that if $X$ and $Y$ are topological space, then $f:X \to Y$ is $\theta$-continuous if and only if $f:(X, \theta_X) \to (Y, \theta_Y)$ is continuous in the sense of convergence spaces.

 Given $A \subseteq X$ and a convergence $\xi$ on $X$, we can define the {\it subconvergence} on $A$ as follows:  If $\F$ is a filter on $A$, let $\hat{\F}$ be the filter on $X$ generated by $\F$.  Define $\lim_{\xi|_A} \F = \lim_\xi \hat{\F} \cap A$.  This is also the initial convergence of $A$ generated by the inclusion map $i:A \to (X, \xi)$; that is, the coarsest convergence making the inclusion map continuous.  Thus, $A$ is a compact subset of $(X, \xi)$ is equivlent to $(A, \xi|_A)$ is a compact convergence space.

\subsection{Pretopologies, H-closed Spaces and H-sets}

For each $x \in X$, the {\it vicinity filter at $x$}, $\mathcal{V}_\xi(x)$ is defined to be $\bigcap \{\F : x \in \lim_\xi \F\}$. A convergence $\xi$ on $X$ is a {\it pretopology} if $\mathcal{V}_\xi(x) \to_\xi x$ for each $x \in X$.  We take a moment to gather several well-known facts and definitions pertaining to pretopological spaces here:

\begin{proposition} If $(X, \pi)$ is a pretopological space, then the adherence operator satisfies each of the following
\begin{enumerate}
\item $adh_\pi \varnothing = \varnothing$,
\item $A \subseteq adh_\pi A$ for each $A \subseteq X$,
\item $adh_\pi (A \cup B) = adh_\pi A \cup adh_\pi B$ for any $A, B \subseteq X$.
\end{enumerate}

Additionally, $U \in \mathcal{V}_\pi(x)$ if and only if $x \in inh_\pi U$ and $x \in adh_\pi \F$ if and only if $\mathcal{V}_\pi(x) \# \F$.
\end{proposition}

In particular, this proposition shows that the categories of \v{C}ech closure spaces and pretopological spaces are equivalent.

\begin{proposition} If $(x, \pi)$ is a pretopological space, then $X$ is Hausdorff if and only if whenever $x_1$, $x_2 \in X$ and $x_1 \not= x_2$, there exists $U_i \in \mathcal{V}_\pi(x_i)$ $(i = 1,2)$ such that $U_1 \cap U_2 = \varnothing$.\end{proposition}

\begin{proposition} Let $f:(X, \pi) \to (Y, \tau)$.  The following are equivalent
\begin{enumerate}
\item $f$ is continuous
\item $f[adh_\pi \F] \subseteq adh_\tau f(\F)$ for each filter $\F$ on $X$
\item $f[adh_\pi A] \subseteq adh_\tau f[A]$ for each $A \subseteq X$
\item $f^\la[inh_\tau B] \subseteq inh_\pi f^\la[B]$ for each $B \subseteq Y$
\item For each $x \in X$, if $V \in \mathcal{V}_\tau(f(x))$, there exists $U \in \mathcal{V}_\pi(x)$ such that $f[U] \subseteq V$.
\end{enumerate}
\end{proposition}

\begin{definition} A collection $\mathcal{C}$ of subsets of a pretopological space $(X, \pi)$ is a {\it $\pi$-cover} (or simply {\it cover} if there is no possible confusion) if for each $x \in X$, $\mathcal{C} \cap \mathcal{V}_\pi(x) \not= \varnothing$. For $A \subseteq X$, we say that $\mathcal{C}$ is a {\it cover of $A$} if for eac $x \in A$, $\mathcal{C} \cap \mathcal{V}_\pi(x) \not= \varnothing.$ \end{definition}

\begin{proposition} Let $(X, \pi)$ be a pretopological space, $\F$ a filter on $X$ and $A \subseteq X$.  Then $\F$ is compact at $A$ if and only if whenever $\mathcal{C}$ is a cover of $A$, there exists $F \in \F$ and $C_1,...,C_n \in \mathcal{C}$ such that $F \subseteq \bigcup_{i=1}^n C_i$.\end{proposition}

The notion of covers has been studied before (see, for example, \cite{Dolecki 2008}) and it is well known that this definition of a pretopological cover is a specific case of the more general notion for convergence spaces.

A familiar example of a pretopology - which is not in general a topology - is given when $X$ is a topological space. In this case, $(X, \theta)$ is a pretopological space and $\mathcal{V}_\xi(x)$ is the filter of closed neighborhoods (in the topological space $X$) at $x$.  The following then characterizes both H-closed spaces and H-sets in the terms of pretopologies.

\begin{theorem} Let $X$ be a Hausdorff topological space and $A \subseteq X$.  
\begin{enumerate}
\item $X$ is H-closed if and only if $(X, \theta_X)$ is a compact pretopological space.
\item $A$ is an H-set in $X$ if and only if $A$ is a compact subset of $(X, \theta_X)$.
\end{enumerate}
\end{theorem}

\begin{proof}
This follow immediately from theorem 2.4 and definition 3.1.
\end{proof}

Just as immediate, but perhaps more interesting, is the case of H-sets in Urysohn spaces.  Recall that a topological space $X$ is Urysohn if distinct points have disjoint closed neighborhoods. The following theorem is due to Vermeer \cite{Vermeer}.

\begin{theorem} Let $X$ be H-closed and Urysohn and $A \subseteq X$.  Then, $A$ is an H-set if and only if $k_X^{\leftarrow}[A]$ is a compact subset of $EX$.\end{theorem}

In the same paper, Vermeer gives an example of an H-closed non-Urysohn space $X$ which has an H-set which is not the image under $k_X$ of any compact subspace of $EX$.  A more general phrasing of the above theorem of Vermeer is that if $A$ is an H-set in an H-closed Urysohn space, then there exists a compact Hausdorff topological space $K$ and a $\theta$-continuous function $f:K \to X$ such that $f[K] = A$.  Vermeer then asked if this was true for an H-set in any Hausdorff topological space; i.e. if $X$ is a Hausdorff topological space, does there exist a compact, Hausdorff topological space $K$ and a $\theta$-continuous function $f:K \to X$ such that $f[K] = A$.  The answer, it turns out, is no.  This was shown first by Bella and Yaschenko in \cite{Bella Yaschenko}.  Later, in \cite{McNeill 2010}, McNeill showed that it is in addition possible to construct a Urysohn space containing an H-set which is not the $\theta$-continuous image of a compact, Hausdorff topological space.  This makes the following observation interesting.

\begin{theorem}
Let $X$ be a Urysohn topological space.  Then, $A$ is an H-set if and only if $(A, \theta|_A)$ is a compact, Hausdorff pretopological space, where $\theta|_A$ is the subconvergence on $A$ inherited from $(X, \theta)$.  In particular, if $X$ is a Urysohn topological space and $A \subseteq X$ is an H-set, then there exists a compact, Hausdorff pretopological space $(K, \pi)$ and a continuous function $f:(K, \pi) \to (X ,\theta)$ such that $f[K] = A$.
\end{theorem}

The question remains - if $X$ is a Hausdorff topological space and $A$ is an H-set in $X$, is there a compact, Hausdorff pretopological space $(K, \pi)$ and a continuous function $f:(K, \pi) \to (X, \theta)$ such that $f[K] = A$?  More broadly, is there a pretopological version of the absolute?  

\subsection{Perfect Maps}

Much of the following can be seen as generalizing the results of \cite{Dickman Porter 1975} to pretopological spaces. Throughout, let $(X, \pi)$ and $(Y, \tau)$ be pretopological spaces. The results below will be used in the construction of the $\theta$-quotient convergence in section 4.

\begin{definition} A function $f:(X, \pi) \to (Y, \tau)$ is {\it perfect} if whenever $\F \to_\tau y$, we have that $f^\leftarrow(\F)$ is compact at $f^\leftarrow(y)$.\end{definition}

In the case of topological spaces, this definition was shown by Whyburn \cite{Whyburn} to be equivalent to the usual definition a perfect function for topological spaces; that is, a function which is closed and has compact fibers.  

\begin{proposition} A function $f:(X, \pi) \to (Y, \tau)$ is perfect if and only if $f(adh_\pi \F) \supseteq adh_\tau f(\F)$ for each filter $\F$ on $X$.\end{proposition}

\begin{proof} 
Suppose that $f$ is perfect.  Let $\F$ be a filter on $X$ and let $y \in adh_\tau f(\F)$.  By way of contradiction, suppose that $f^\leftarrow(y) \cap adh_\pi \F = \varnothing$.  Since $\tau$ is a pretopology, $\mathcal{V}_\tau(y) \to_\tau y$ and since $f$ is perfect, it follows that $f^\leftarrow(\mathcal{V}_\tau(y))$ is compact at $f^\leftarrow(y)$.  Since $y \in adh_\tau f(\F)$, $\mathcal{V}_\tau(y) \# f(\F)$.  It follows that $f^\leftarrow(\mathcal{V}_\tau(y)) \# \F$.  Thus, it must be that $adh_\pi \F \cap f^\leftarrow(y) \not= \varnothing$, a contradiction.  Hence, $y \in f(adh_\pi \F)$. 

Conversely, suppose $\F$ is a filter on $Y$ and $\F \to_\tau y$. Let $\G$ be a filter on $X$ such that $\G \# f^\leftarrow(\F)$.  Then $f(\G) \# \F$.  Since $\F \to_\tau y$, it follows that $y \in adh_\tau f(\G) \subseteq f[adh_\pi \G]$.  So, we can find $x \in adh_\pi \G$ such that $f(x) = y$.  In other words, $adh_\pi \G \cap f^\leftarrow(y) \not= \varnothing$, and $f^\leftarrow(\F)$ is compact at $f^\leftarrow(y)$.
\end{proof}

To get a similar characterization to that of perfect functions between topological spaces for perfect functions between pretopological spaces we need the concept of {\it cover-compact} sets, a strengthening of compact sets. This characterization can be found in \cite{Dolecki 2002}, but we feel it is worthwhile to lay out the details in this less technical setting.

\begin{definition} 
Let $(X, \pi)$ be a pretopological space and $A \subseteq X$.  Then $A$ is {\it cover-compact} whenever $\mathcal{C}$ is a cover of $A$, there exist $C_1, ..., C_n \in \mathcal{C}$ such that $A \subseteq inh_\pi \left(\bigcup_{i=1}^n C_i\right)$.
\end{definition}

\begin{proposition} 
Let $(X, \pi)$ be a pretopological space and $A \subseteq X$.  The following are equivalent,
\begin{enumerate}
\item $A$ for any filter $\F$ on $X$, $adh_\pi \F \cap A = \varnothing$ implies that there exists some $F \in \F$ such that $adh_\pi F \cap A = \varnothing$,
\item $A$ is cover-compact,
\item $adh_\pi \F \cap A = \varnothing$ implies there exists $V \subseteq X$ and $F \in \F$ such that $A \subseteq inh_\pi V$ and $V \cap F = \varnothing$ for and filter $\F$ on $X$.
\end{enumerate}
\end{proposition}

\begin{proof}
Suppose that $A$ is cover-compact and let $\mathcal{C}$ be a cover of $A$. Suppose that no finite subcollection exists as needed.  Then $\F = \{X \setminus (C_1 \cup ... \cup C_n) : C_i \in \mathcal{C}, i \in \naturals\}$ is a filterbase on $X$.  Note that $adh_\pi \F \subseteq X \setminus \bigcup_{C \in \mathcal{C}} inh_\pi C$ and as such $adh_\pi \F \cap A = \varnothing$.  Since $A$ is cover-compact, we can find $F \in \F$ such that $adh_\pi F \cap A = \varnothing$.  However, $F = X \setminus (C_1 \cup ... \cup C_n)$ for some $C_1,...,C_n \in \mathcal{C}$, so we have that $A \subseteq inh_\pi (C_1 \cup ... \cup C_n)$, a contradiction.

Suppose that $\mathcal{F}$ is a filter on $X$ and $adh_\pi \F \cap A = \varnothing$.  Then, for each $x \in A$, fix $V_x \in \mathcal{V}_\pi(x)$ and $F_x \in \mathcal{F}$ such that $V_x \cap F_x = \varnothing$.  Then $\{V_x : x \in A\}$ is a cover of $A$.  By assumption, we can choose $x_1,...,x_n \in A$ such that $A \subseteq inh_\pi \left(\bigcup_{i=1}^n V_{x_i} \right)$.  Therefore, $V = \bigcup_{i=1}^n V_{x_i} \in \mathcal{V}_\pi(A)$ and $V \cap (F_{x_1} \cap ... \cap F_{x_n}) = \varnothing$.  Since $F_{x_1} \cap ... \cap  F_{x_n} \in \mathcal{F}$, we have shown that (c) holds.

Lastly, let $\mathcal{F}$ be a filter on $X$ such that $adh_\pi \F \cap A = \varnothing$.  By assumption, we can find $V \in \mathcal{V}_\pi(A)$ and $F \in \F$ such that $V \cap F = \varnothing$.  For each $x \in A$, $V \in \mathcal{V}_\pi(x)$, so $x \notin adh_\pi F$.  It follows immediately that $A \cap adh_\pi F \not= \varnothing$.
\end{proof}

It is useful to note that if $A \subseteq X$ is cover-compact, then $adh_\pi A = A$.

\begin{theorem} Let $f:(X, \pi) \to (Y, \tau)$ be a map between pretopological spaces satisfying (a) $f[adh_\pi A] \supseteq adh_\tau f[A]$ for any $A \subseteq X$ and (b) $f^\leftarrow(y)$ is cover-compact for each $y \in Y$.  Then, $f$ is perfect. 
\end{theorem}

\begin{proof}
Let $\mathcal{F}$ be a filter on $Y$ which $\tau$-converges to some $y \in Y$.  Let $\mathcal{G}$ be a filter on $X$ which meets $f^\la(\F)$.  Then, $f(\G)$ meets $\F$.  Since $\F \to_\tau y$, $\F$ is compact at $y$.  Therefore, $y \in adh_\tau f(\G) = \bigcap_{G \in \G} adh_\tau f[G]$.  By assumption (a), for each $G \in  \G$, $f[adh_\pi G] \supseteq adh_\tau f[G]$.  Therefore, $adh_\pi G \cap f^\la(y) \not= \varnothing$ for each $G \in \G$.  By assumption (b), $f^\la(y)$ is cover-compact, so $adh_\pi \G \cap f^\la(y) \not= \varnothing$.  In other words, $f^\la(F)$ is compact at $f^\la(y)$ and $f$ is perfect.
\end{proof}

\begin{theorem} 
Let $f:(X,\pi) \to (Y, \tau)$ be perfect and continuous.  Then, $f$ satisfies (a) and (b) of 4.13.
\end{theorem}

\begin{proof}
By proposition 4.8(c) and proposition 4.10, $f[adh_\pi A] = adh_\tau f[A]$ for each $A \subseteq X$.  Thus, a property stronger than (a) holds. To see that (b) holds, fix $y \in Y$ and let $\F$ be a filter on $X$ such that $adh_\pi \F \cap f^\leftarrow(y) = \varnothing$. By proposition 4.10, $y \notin f[adh_\pi \F] \supseteq adh_\tau f(\F)$.  Thus, we can find $V \in \mathcal{V}_\tau(y)$ and  $F \in \F$ such that $V \cap f(F) = \varnothing$. It follows that $f^\la[V] \cap F = \varnothing$. Since $f$ is a continuous function, for each $x \in f^\leftarrow(y)$, fix $U_x \in \mathcal{V}_\pi(x)$ such that $f[U_x] \subseteq V$.  Then, $\bigcup_{x \in f^\la(y)} U_x \subseteq f^\la[V]$ and thus $\bigcup_{x \in f^\la(y)} U_x \cap F = \varnothing$.  So, $adh_\pi F \cap f^\la(y) = \varnothing$, as needed.
\end{proof}

\begin{corollary} 
A continuous function $f:(X, \pi) \to (Y, \tau)$ is perfect if and only if it satisfies (a) and (b) of 4.13.
\end{corollary}

\section{\bf PHC Spaces}

In this section we will define a generalization of H-closed spaces to pretopological spaces.  After establishing some basic facts about the so-called {\it PHC spaces}, we will describe a method for constructing PHC pretopologies and PHC extensions.

The following definition appears in \cite{Dolecki Gauld}.

\begin{definition} Let $(X, \pi)$ be a pretopological space.  The {\it partial regularization} $r \pi$ of $\pi$ is the pretopology determined by the vicinity filters $\mathcal{V}_{r\pi}(x) = \{adh_\pi U : U \in \mathcal{V}_\pi(x)\}$.  \end{definition}

Notice that if $(X, \tau)$ is a topological space, then $r\tau$ is the usual $\theta$-convergence on $X$.  Thus, a Hausdorff topological space $(X, \tau)$ is H-closed if and only if $(X, r\tau)$ is compact.  This inspires the following definition, aiming to generalize the notion of H-closed spaces to pretopological spaces.

\begin{definition} A Hausdorff pretopological space $(X, \pi)$ is {\it PHC (pretopologically H-closed)} if $(X, r\pi)$ is compact.  Without the assumption of Hausdorff, we will use the term {\it quasi PHC} \end{definition}

Given a filter $\F$ on a pretopological space $(X, \pi)$ let $\F^1 = \{F \in \F : inh_\pi F \in \F\}$.  Inductively, define $\F^n = \{F \in \F^{n-1} : inh_\pi F \in \F^{n-1}\}$.  Define $\F^\circ  = \bigcap_{n \in \naturals} \F^n$.  If we define a filter $\F$ to be {\it pretopologically open} if $F \in \F$ implies $inh_\pi F \in \F$, then $\F^\circ$ is the largest pretopologically open filter contained in $\F$.

\begin{lemma} Let $(X, \pi)$ be a pretopological space and let $\F$ be a pretopologically open filter on $X$.  Then, $adh_\pi \F = adh_{r\pi} \F$. \end{lemma}

\begin{proof} To begin, since $r\pi < \pi$, $adh_\pi \F \subseteq adh_{r\pi} \F$.  Now, $x \notin adh_\pi \F$ if and only if we can find $F \in \F$ and $U \in \mathcal{V}_\pi(x)$ such that $U \cap F = \varnothing$.  Since $U \cap F = \varnothing$, if $y \in inh_\pi F$, then $y \notin adh_\pi U$.   In other words, $adh_\pi U \cap inh_\pi F = \varnothing$.  Since $\F$ is open, $inh_\pi F \in \F$ and by definition $x \notin adh_{r\pi} \F$, as needed. \end{proof}

\begin{lemma} Let $(X, \pi)$ be a pretopological space and let $\F$ be a filter on $X$.  Then, $adh_{r\pi} \F = adh_\pi \F^1$.  More, for each $n \in \naturals$, $adh_{r\pi} \F^n = adh_\pi \F^{n+1}$.\end{lemma}

\begin{proof} Suppose that $x \notin adh_{r\pi} \F$.  Then there exists $U \in \mathcal{V}_\pi(x)$ and there exists $F \in \F$ such that $adh_\pi U \cap F = \varnothing$.  So, $F \subseteq X \setminus adh_\pi U = inh_\pi(X \setminus U)$.  By definition, it follows that $X \setminus U \in \F^1$.  Since $U \cap X \setminus U = \varnothing$, we have that $x \notin adh_\pi \F^1$.  Conversely, if $x \notin adh_\pi \F^1$, then there exists $U \in \mathcal{V}_\pi(x)$ and  $F \in \F^1$ such that $U \cap F = \varnothing$.  As we have seen before, it follows that $adh_\pi U \cap inh_\pi F = \varnothing$.  Since $F \in F^1$, we know that $inh_\pi F \in \F$.  It follows that $x \notin adh_{r\pi} \F$, as needed.  

The remainder of the lemma follows easily by setting $\F = \F^n$, in which case $\F^1 = \F^{n+1}$. \end{proof} 

\begin{definition} A filter $\F$ on a pretopological space is {\it inherent} if $inh_\pi F \not= \varnothing$ for each $F \in \F$.  If $\mathcal{U}$ is maximal with respect to the property of being inherent, we say that $\mathcal{U}$ is an {\it inherent ultrafilter}.\end{definition}

\begin{theorem} For a Hausdorff pretopological space $(X, \pi)$, the following are equivalent.
\begin{enumerate}
\item $X$ is PHC
\item whenever $\mathcal{C}$ is a $\pi$-cover of $X$, there exists $C_1,...,C_n \in \mathcal{C}$ such that $X = \bigcup_{i=1}^n adh_\pi C_i$
\item each inherent filter $\F$ on $X$ has nonempty adherence
\item for each filter $\F$ on $X$, $adh_\pi \F^1 \not= \varnothing$.
\end{enumerate}
\end{theorem}

\begin{proof} Let $\mathcal{C}$ be a $\pi$-cover of $X$. Without loss of generality, assume that $\mathcal{C} = \{U_x : x \in X\}$ where each $U_x \in \mathcal{V}_\pi(x)$. Suppose no such finite subcollection exists.  Then, $\mathcal{A} = \{X \setminus adh_\pi U_x : x \in X\}$ has fip.  Let $\F$ be the filter generated by $\mathcal{A}$.  For each $x \in X$, $x \in inh_{r\pi} adh_\pi U$ if and only if there exists $V \in \mathcal{V}_\pi(x)$ such that $adh_\pi V \subseteq adh_\pi U$.  Therefore, for each $x \in X$, $x \in inh_{r\pi}adh_\pi U_x$.  Thus, $adh_{r\pi} \F = X \setminus \bigcup_{x \in X} inh_{r\pi} adh_\pi U_x = \varnothing$, a contradiction.

Next, let $\F$ be a filter on $X$ such that $inh_\pi F \not= \varnothing$ for each $F \in \F$.  Suppose that $adh_\pi \F = \varnothing$.  Then, $\mathcal{C} = \{X \setminus F : F \in \F\}$ is a $\pi$-cover of $X$.  By assumption, there exist $F_1,...,F_n \in \mathcal{F}$ such that $adh_\pi(X \setminus F_1 \cup ... \cup X \setminus F_n) = X \setminus inh_\pi (F_1 \cap ... \cap F_n) = X$.  However, $F_1 \cap ... \cap F_n \in \mathcal{F}$ and thus by assumption $F_1 \cap ... \cap F_n$ has nonempty inherence, a contradiction.

Let $\F$ be a filter on $X$.  Notice that $\F^1$ is a filter on $X$ such that $inh_\pi F \not= \varnothing$ for each $F \in \mathcal{F}^1$.  Then, by assumption, $adh_\pi \F^1 \not= \varnothing$.  By the lemma,

Let $\F$ be a filter on $X$.  Then, $adh_{r\pi} \F = adh_{\pi} \F^1 \not= \varnothing$ by lemma 4.4.  Thus, we have shown that $(X, r\pi)$ is compact and the theorem is proven.\end{proof}

\subsection{$\theta$-quotient Convergence}

Let $(X, \pi)$ be a compact Hausdorff pretopological space, $Y$ a set and $f: (X, \pi) \to Y$ a surjection such that $f^\leftarrow(y)$ is cover-compact for each $y \in Y$.  For $A \subseteq X$, let $f^\#[A] = \{y \in Y : f^\la(y) \subseteq A\}$.  Define the {\it $\theta$-quotient convergence} $\sigma$ on $Y$ as follows: a filter $\F$ on $Y$ $\sigma$-converges to $y$ if and only if $f^\leftarrow(\F)$ is compact at $f^\leftarrow(y)$.

\begin{lemma} Let $\F$ be a filter on $Y$.  Then $\F \to_\sigma y$ if and only if $f^\leftarrow(\F) \supseteq \mathcal{V}_\pi(f^\leftarrow(y))$. \end{lemma}

\begin{proof} 
Suppose that $f^\la(\F)$ is compact at $f^\la(y)$.  Then, whenever $\mathcal{C}$ is a cover of $f^\la(y)$.  Then, there exists $F \in \F$ and $C_1,...,C_n \in \mathcal{C}$ such that $f^\la[F] \subseteq \bigcup_{i=1}^n C_i$.  Let $V \in \mathcal{V}_\pi(f^\la(y))$.  By definition, $f^\la(y) \subseteq inh_\pi V$.  In other words, $\{V\}$ is a one-element cover of $f^\la(y)$.  Thus, $V \in f^\la(F)$, as needed.

Conversely, let $\mathcal{C}$ be a cover of $f^\la(y)$.  Since $f^\la(y)$ is cover-compact, we can find $C_1,...,C_n \in \mathcal{C}$ such that $f^\la(y) \subseteq inh_\pi\left(\bigcup_{i=1}^n C_i\right)$.  By definition, $C = \bigcup_{i=1}^n C_i \in \mathcal{V}_\pi(f^\la(y))$.  Thus, $C \in f^\la(F)$ and there exists $F \in \F$ such that $f^\la[F] \subseteq \bigcup_{i=1}^n C_i$ and $f^\la(\F)$ is compact at $f^\la(y)$.
\end{proof}

\begin{lemma} Let $(X, \pi)$ be a Hausdorff pretopology. If $A, B \subseteq X$ are disjoint cover-compact subsets of $X$, then there exist disjoint vicinities $U \in \mathcal{V}_\pi(A)$, $V \in \mathcal{V}_\pi(B)$.\end{lemma}

\begin{proof}
First we show this holds for $B = \{x\}$.  For each $z \in A$, choose disjoint $U_z \in \mathcal{V}_\pi(z)$ and $V_z \in \mathcal{V}_\pi(x)$.  Since $A$ is cover-compact, by proposition 4.13(b), we can choose $z_1,...,z_n \in A$ such that $A \subseteq inh_\pi\left(\bigcup_{i=1}^n U_{z_i}\right)$.  Thus, $U = \bigcap_{i=1}^n U_{z_i} \in \mathcal{V}_\pi(A)$.  Also, $V = \bigcap_{i=1}^n V_{z_i} \in \mathcal{V}_\pi(x)$ and $U \cap V = \varnothing$.  Again using proposition 4.13(b), it is a straight-forward exercise to now show this holds for disjoint cover-compact sets, $A$ and $B$.
\end{proof}

\begin{proposition} $(Y, \sigma)$ is a Hausdorff pretopology.  Further, for each $y \in Y$, $\mathcal{V}_\pi(y)$ is the filter generated by $\{f^\#[W] : W \in \mathcal{V}_\pi(f^\leftarrow(y))\}$. \end{proposition}

\begin{proof} We first show that $\sigma$ is indeed a pretopology.  Notice that for $y \in Y$, $\bigcap \{\F : \F \to_\sigma y\} = \bigcap\{\F : V \in \mathcal{V}_\pi(f^\la(y)) \text{ implies } f^\#[V] \in \F\}$.  It follows that $\mathcal{V}_\sigma(y)$ is the filter generated by $\{f^\#[U] : U \in \mathcal{V}_\pi(f^\la(y))\}$.  For any $A \subseteq X$, $f^\la[f^\#[A]] \subseteq A$.  It follows easily that $f^\la(\mathcal{V}_\sigma(y)) \supseteq \mathcal{V}_\pi(f^\la(y))$.  By lemma 4.16, then, $\mathcal{V}_\sigma(y) \to_\sigma y$ and $\sigma$ is a pretopology with the stated vicinity filters.

Now, if $y_1 \not= y_2$, by lemma (4.17), for $i = 1,2$, we can find $U_i \in \mathcal{V}_\pi(f^\la(y_i))$ such that $U_1 \cap U_2 = \varnothing$.  It is immediate that $f^\#[U_1] \cap f^\#[U_2] = \varnothing$ and $\sigma$ is Hausdorff.\end{proof}

\begin{definition} Let $U$ and $V$ be subsets of $X$ with nonempty inherence.  A function $f:(X, \pi) \to (Y, \tau)$ between pretopological spaces is {\it strongly irreducible} if whenever $U \cap V \not= \varnothing$, there exists $y \in Y$ such that $f^\leftarrow(y) \subseteq U \cap V$.  

The function $f$ is {\it weakly $\theta$-continuous (w$\theta$-continuous} for short) if $f:(X, \pi) \to (Y, r\tau)$ is continuous.\end{definition}


\begin{theorem} If $(X, \pi)$ is a compact, Hausdorff pretopological space, $f:(X, \pi) \to Y$ an strongly irreducible, surjection such that $f^\leftarrow(y)$ is cover-compact for each $y \in Y$ and $\sigma$ is the $\theta$-quotient pretopology on $Y$, then $f:(X, \pi) \to (Y, \sigma)$ is w$\theta$-continuous and $(Y, \sigma)$ is a PHC Hausdorff pretopological space.\end{theorem}

\begin{proof}
For $x \in X$, let $V \in \mathcal{V}_\sigma(f(x))$.  Without loss of generality, we can assume that $V = f^\#[W]$ for some $W \in \mathcal{V}_\pi(f^\la(f(x)))$.  Note that in this case $x \in inh_\pi W$.  Supose that $w \in W$ and $f(w) \in f^\#[U]$ for some $U \in \mathcal{V}_\pi(f^\la(f(w)))$.  Notice that $w \in W \cap U$, so $W \cap U \not= \varnothing$.  Since $f$ is strongly irreducible, we can find $y \in f^\#[U] \cap f^\#[W]$.  Therefore, $f(w) \in adh_\pi f^\#[W]$.  In particular, $f[W] \subseteq adh_\pi f^\#[W]$ and $f$ is w$\theta$-continuous.

Since the continuous image of a compact space is again compact, $(Y, r\sigma)$ is compact and by definition $(Y, \sigma)$ is PHC.
\end{proof}

\subsection{PHC Extensions of X}

By an extension of $X$, we mean a convergence space $(Y, \xi)$ which contains $X$ as a subspace such that $adh_\xi X = Y$.  There is an ordering on the family extensions of $X$.  If $(Y,\xi)$ and $(Z, \zeta)$ are extensions of $X$, we say that {\it $Y$ is projectively larger than $Z$}, written $Y \geq_X Z$ if there exists a continuous map $f:(Y, \xi) \to (Z, \zeta)$ which fixes the points of $X$.

We borrow from topology the concepts of {\it strict} and {\it simple} extensions.  If $Y$ is an extension of $(X,\pi)$, we define $Y^+$ a new extension of $X$ on the same underlying set.   For $p \in Y^+$, $\V_+(p)$ is the filter generated by $\{\{p\} \cup U : \exists W \in \V_Y(p), W \cap X = U\}$.  If $Y = Y^+$, then we say $Y$ is a {\it strict extension} of $X$.

In a similar way, we define $Y^\#$, an extension of $X$ on the same set as $Y$.  If $A \subseteq X$, let $o(A) = \{p \in Y : \exists W \in \V_Y(p), W \cap X = A\}$.  If $p \in Y^\#$, then $V_\#(p)$ is the filter generated by $\{oA : \exists V \in \V_Y(p), V \cap X = A\}$.

\begin{lemma} If $Y$ is an extension of $X$, then $Y^\# \leq Y \leq Y^+$.\end{lemma}

\begin{proof}
In both cases it is straight-forward to check that the identity map is continuous and fixes $X$.
\end{proof}

\begin{proposition} Suppose that $(X, \pi)$ is a Hausdorff pretopological space and $Y$ is a compactification of $X$.  Then $Y^+$ is PHC. \end{proposition}

\begin{proof}
Fix $p \in Y^+$ and let $\{p\} \cup U$ be a vicinity of $p$ in $Y^+$.  Then, $adh_{Y^+}(\{p\} \cup U) = oU \cup adh_\pi U$. So, in the partial regularization of $Y^+$, the vicinity filters are generated by sets of the form $oU \cup adh_\pi U$ for $U \subseteq X$.  In particular, this shows that $\mathcal{V}_{rY^+}(p) \subseteq \mathcal{V}_{Y^\#}(p)$ for each $p \in Y$.  Since $Y^\#$ has a coarser pretopology than $Y$, it follows that the partial regularization of $Y^+$ is coarser than $Y$.   Since $Y$ is compact, so is $rY^+$ and by definition, $Y^+$ is PHC.
\end{proof}


For any Hausdorff convergence space $X$, Richardson \cite{Richardson 1970} constructs a compact, Hausdorff convergence space $X^*$ in which $X$ is densely embedded.  If $X$ is a pretopology, then so is $X^*$.  It is said that $X$ is {\it regular} if $\F \to x$ implies that $\{adh F : F \in \F\} \to x$.  Richardson \cite{Richardson 1970} proves the following:

\begin{theorem} If $X$ is a Hausdorff convergence space, $Y$ is a compact, Hausdorff, regular convergence space and $f:X \to Y$ is continuous, then there exists a unique continuous map $F:X^* \to Y$ extending $f$. \end{theorem}

We seek to circumvent the assumption of regularity on $Y$.  For a Hausdorff pretopological space $X$, let $\kappa X = (X^*)^+$.  By the above proposition, $\kappa X$ is PHC.  Additionally, $\kappa X$ has the following property.

\begin{theorem} If $f:(X,\pi) \to (Y, \tau)$ is continuous, then there exists a continuous function $F:\kappa_\pi X \to \kappa_\tau Y$ which extends $f$. \end{theorem}
  
\begin{proof}
For each free ultrafilter $\U$ on $X$, $f(\U)$ is an ultrafilter on $Y$. Define $F(\U)$ as follows:
\begin{itemize}
\item If $f(\U) \to_\tau y$ for some $y \in Y$, let $F(\U) = y$.
\item If $f(\U)$ is free in $(Y, \tau)$, let $F(\U) = f(\U)$.
\end{itemize}

We show that $F$ is continuous.  Since $f$ is continuous, if $x \in X$ and $F(x) \in \V_{\kappa\pi}(f(x)) = \V_\pi(f(x))$, then we can find $U \in \V_{\pi}(x)$ such that $f[U] \subseteq V$.  Suppose $\U \in X'$.  If $F(\U) \in Y$, let $V \in \mathcal{V}_\pi(F(\U))$.  Since $f(\U) \to_\tau y$, $V \in f(\U)$.  Therefore, for some $U \in \U$, $f(U) \subseteq V$.  It follows that $F[\{\U\} \cup U] \subseteq V$.  Lastly, suppose that $F(\U) \in Y'$ and fix $V \in F(\U) = f(\U)$.  Then, for some $U \in \U$, $f[U] \subseteq V$.  So, $F[\{U\} \cup U] \subseteq \{F(\U)\} \cup V$ and $F$ is continuous.
\end{proof} 

In \cite{Kent Richardson 1979}, it is shown that a convergence $X$ has a projective maximum compactification if and only if $X$ has only finitely many free ultrafilters.  In contrast with this, we have the following corollaries:

\begin{corollary} If $X$ is a pretopological space, $Y$ is a compact pretopological space and $f:X \to Y$ is continuous, then there exists a continuous map $F:\kappa_\pi X \to Y$ extending $f$.\end{corollary}

\begin{corollary} If $X$ is a pretopological space, then $\kappa_\pi X \geq Y$ for any compactification $Y$ of $X$.\end{corollary}

\bibliographystyle{plain}

\end{document}